\documentclass[12pt,leqno]{article}
\usepackage{amsmath,amsthm,amsfonts,amssymb}
\makeatletter
\def\@seccntformat#1{\csname the#1\endcsname.\quad}
\renewcommand\section{\@startsection {section}{1}{\z@}%
                     {-3.5ex \@plus -1ex \@minus -.2ex}%
                     {2.3ex \@plus.2ex}%
                     {\normalfont\large\bf\centering}}
\makeatother
\numberwithin{equation}{section}

\theoremstyle{plain}
\newtheorem{theorem}[equation]{Theorem}
\newtheorem{lemma}[equation]{Lemma}

\theoremstyle{definition}

\newtheorem{example}[equation]{Example}

\theoremstyle{remark}
\newtheorem{remark}[equation]{Remark}
\newtheorem*{remark*}{Remark}

\newcommand{\kyosu}{\sqrt{-1}\,}

\newcommand{\bysame}{\underline{\hspace{1.5cm}}}
\newcommand{\llangle}{\langle\!\langle}
\newcommand{\rrangle}{\rangle\!\rangle}
\renewcommand{\Re}{\mathrm{Re}}
\renewcommand{\Im}{\mathrm{Im}}

\title{A construction of diffusion processes associated with sub-Laplacian 
on CR manifolds and its applications}
\author{ Hiroki Kondo 
 \thanks{\ E-mail : h-kondo@math.kyushu-u.ac.jp}
\\ {\small\textit{Graduate School of Mathematics, Kyushu University}}
\and
 Setsuo Taniguchi 
 \thanks{\ E-mail : se2otngc@artsci.kyushu-u.ac.jp}
\\ {\small\textit{Faculty of Arts and Science, Kyushu University}}}
\date{}
\begin{document}
\maketitle

\begin{abstract}
A diffusion process associated with the real sub-Laplacian
$\Delta_b$, the real part of the complex Kohn-Spencer laplacian
$\square_b$, on a strictly pseudoconvex CR manifold is
constructed via 
the Eells-Elworthy-Malliavin method
by taking advantage of the metric connection due to Tanaka-Webster.  Using
the diffusion process and the Malliavin calculus, 
the heat kernel and the Dirichlet problem for $\Delta_b$ are studied
in a probabilistic manner. 
Moreover, distributions of stochastic line integrals along
the diffusion process will be investigated.
\end{abstract}

\section*{Introduction}
Let $M$ be an oriented strictly pseudoconvex CR manifold of
dimension $2n+1$, 
and $\Delta_b$ be the real sub-Laplacian on $M$,
i.e. the real part of the Kohn-Spencer laplacian $\square_b$. 
For definitions, see Section~\ref{sec.cr.geometry}.
The first aim of this paper is to construct the diffusion process generated by $-\Delta_b/2$ 
by extending the Eells-Elworthy-Malliavin method \cite{elworthy,malliavin}, 
namely we consider stochastic development of the Brownian motion on $\mathbb{C}^n$ in
the complex unitary bundle of $M$. 
The second aim is to apply the diffusion processes; 
the heat kernel and 
the Dirichlet problem associated with
$\Delta_b$ will be studied in a probabilistic manner 
with the help of the Malliavin calculus.
Moreover, distributions of stochastic line integrals along
the diffusion process will be investigated by the
partial hypoelliptic argument.

The Eells-Elworthy-Malliavin method is one of constructions of the Brownian motion 
on a Riemannian manifold, and realizes the Brownian motion as the 
projection of the solution of the 
stochastic differential equation (SDE in abbreviation) on 
the orthonormal frame bundle over the Riemannian manifold. 
See, for example, \cite{elworthy, hsu, i-w, malliavin, stroock}. 
We will carry out this method on a CR manifold, 
but this time used is a complex unitary frame bundle instead of 
a real orthonormal frame bundle. 
This comes from that the CR structure is defined as a 
complex subbundle $T_{1,0}$ of the complexified tangent bundle $\mathbb{C}TM$. 

To be more precise, recall that in the Eells-Elworthy-Malliavin method on a
Riemannian manifold, the vector fields governing the SDE on
the orthonormal bundle are constructed with the help of the
Riemannian connection.
The SDE corresponds to the stochastic parallel translation
on the Riemannian manifold.
In our constuction on a CR manifold, we take advanage of the
metric connection on the complex subbundle $T_{1,0}$ due to 
Tanaka \cite{tanaka} and Webster \cite{webster} to have vector fields 
$L_1,\ldots,L_n$ on the unitary bundle $U(T_{1,0})$ over $M$.
Solving the SDE on $U(T_{1,0})$ governed by $L_1,\ldots,L_n$ and
projecting its solution onto $M$, we arrive at the diffusion
process $\mathbb{X}=\{(\{X(t)\}_{t\ge0},P_x);x\in M\}$ on $M$ 
generated by $-\Delta_b/2$.
See Section~\ref{sec.diffusion}.

By using the partial hypoelliplicity argument in the Malliavin calculus, 
we will obtain the heat kernel related to this diffusion process, 
that is, we will show the transition probability function of $\mathbb{X}$ has 
a smooth density function $p(t,x,y)$.
Moreover, we will give a sufficient condition for distributions of stochastic line integrals 
of 1-forms on $M$ along the diffusion process $\mathbb{X}$ to 
have smooth density functions.
See Section~\ref{sec.heat}. 

We finally consider the Dirichlet problem associated with $\Delta_b$.  
Let $G$ be a relatively compact open set in $M$ with $C^3$-boundary.
We shall show in a probabilistic manner that, for
each $f\in C(\partial G)$, there is a $u\in C(\overline{G})$ such that 
\begin{equation}\label{eqn.dirichlet}
 \Delta_b u=0 \quad\text{on $G$ in the weak sense,}\quad 
\text{and $u=f$ on $\partial G$.}
\end{equation}
See Theorem \ref{thm.dirichlet}. 
As will be seen in Remark \ref{rem.hypoelliptic}, 
together with hypoelliplicity of $\Delta_b$, 
this $u$ is a classical solution to the Dirichlet problem. 
In the proof, a key role is played by
the local representation of the sub-Laplacian
$\Delta_b$ so that, on every sufficiently small 
coordinate neighborhood $U$, there are $a^\alpha\in {\mathbb C}$
and $C^\infty$-vector fields $Z_\alpha$ with ${\mathbb C}$-valued
coefficients on $U$ so that
\begin{equation}\label{eqn.loc.decomp}
 \Delta_b = - \sum_{\alpha=1}^n (Z_\alpha
 Z_{\overline{\alpha}} 
 + Z_{\overline{\alpha}} Z_\alpha)
 + \sum_{\alpha=1}^n
 (a^\alpha Z_\alpha  + a^{\overline{\alpha}}
 Z_{\overline{\alpha}}),
\end{equation}
and 
\begin{equation}\label{eqn.span}
 \mathrm{span}_{{\mathbb C}} \{ (Z_\alpha)_x,
 (Z_{\overline{\alpha}})_x, 
 [Z_\alpha,Z_{\overline{\alpha}}]_x ; 1\le \alpha\le n
 \} = \mathbb{C} T_xM, \quad x\in U,
\end{equation}
where $[\cdot,\cdot]$ denotes the Lie bracket product, $T_xM$ is
the tangent space of $M$ at $x$, and we have used super and
subscripts $\overline{\alpha}$'s to indicate that complex conjugates
are taken; $b_{\overline{\alpha}}=\overline{b_\alpha}$, and 
$c^{\overline{\alpha}} = \overline{c^\alpha}$.

In Section~\ref{sec.cr.geometry}, we shall give a brief review on CR
geometry.  In the same section, we shall construct vector fields on
the complex bundle $U(T_{1,0})$ over $M$, which are associated with
the metric connection due to Tanaka-Webster. 
These vector fields will be
used in Section~\ref{sec.diffusion} to construct a diffusion process
$\mathbb{X}$ generated by $-\Delta_b/2$. 
The heat kernel and distributions of stochastic line integrals 
along $\mathbb{X}$ are studied in Section~\ref{sec.heat}.
Section~\ref{sec.dirichlet} will be devoted to the study of
Dirichlet problems associated with $\Delta_b$. 

\section{CR geometry}\label{sec.cr.geometry}

\subsection{CR manifolds}
We begin this section with listing the results on CR manifolds 
which we shall use later, 
following Dragomir-Tomassini \cite{dragomir-tomassini} and Lee \cite{lee}. 

A CR manifold $M$ is a real differentiable manifold together with
a complex subbundle $T_{1,0}$ of the complexified tangent bundle
$\mathbb{C}TM=TM\otimes_{\mathbb{R}}\mathbb{C}$ such that $T_{1,0}\cap T_{0,1}=\{0\}$ and
$[T_{1,0},T_{1,0}]\subset T_{1,0}$, where $T_{0,1}=\overline{T_{1,0}}$. 
We consider the case that $M$ is orientable and of real dimension $2n+1$ with
$n\in\mathbb{N}=\{1,2,3,\ldots\}$ and $T_{1,0}$ is of complex dimension $n$, 
i.e. the CR codimension is 1.  

Set $H=\Re(T_{1,0}\oplus T_{0,1})$, which is called 
the Levi distribution of $(M,T_{1,0})$. 
There exists a pseudo-Hermitian structure, that is, 
a real non-vanishing $1$-form $\theta$ on $M$ which annihilates $H$.
For such $\theta$, the Levi form $L_\theta$ of $\theta$ is defined by 
\[
 L_\theta(Z,W) = -\kyosu d\theta(Z,W),
 \quad Z,W\in \Gamma^\infty(T_{1,0}\oplus T_{0,1}),
\]
where
$\Gamma^\infty(V)$ stands for the space of
$C^\infty$ cross sections of a vector bundle $V$.  
Throughout the paper, we assume that $M$ is strictly pseudoconvex, 
that is, the Levi form $L_\theta$ is positive definite. 
Then $T_{1,0}$ is an Hermitian fiber bundle with Hermitian fiber metric $L_\theta$.
Let $T$ be the characteristic direction, that is, 
the unique real vector field on $M$ transverse to $H$, defined by
\begin{equation}\label{eqn.2.1}
 T \rfloor d\theta=0, \quad T\rfloor\theta=1,
\end{equation}
where $T\rfloor \omega$ is the interior product:
$T\rfloor
\omega(X_1,\dots,X_{p-1})=\omega(T,X_1,\dots,X_{p-1})$ for
a $p$-form $\omega$.

As $M$ is strictly pseudoconvex, the $(2n+1)$-form
$\psi=\theta\wedge(d\theta)^n$ on $M$ determines a volume form, 
where we have chosen the orientation of $M$ so that
$\psi$ is a positive form. 
Then it induces the $L^2$-inner product on functions:
\[
 \langle u,v\rangle_\theta = \int_M u\overline{v}\psi, 
\quad u, v\in C_0^{\infty}(M;\mathbb{C})
\equiv \{f+\kyosu g;f,g\in C_0^{\infty}(M)\}.
\]
The Levi form induces a metric
on $H$ (denoted by $L_\theta$ again), and the dual metric $L_\theta^*$ on $H^*$. 
Then the $L^2$-inner product on sections of $H^*$ is
given by 
\[
 \langle \omega,\eta \rangle_\theta
 = \int_M L_\theta^*(\omega,\eta) \psi, \quad
\omega, \eta\in\Gamma^{\infty}(H^*).
\]
Denoting by $r\colon T^*M\to H^*$ the natural restriction mapping, 
we define a section $d_b u$ for $u\in C^\infty(M)$ by 
$d_b u = r \circ du$. 
The real sub-Laplacian $\Delta_b$ on functions is given by
\[
 \langle \Delta_b u,v \rangle_\theta
 = \langle d_b u, d_b v \rangle_\theta,
 \quad v\in C_0^\infty(M).
\]
Similarly, denoting by $\overline{\partial}_b u$ the projection of
$du$ onto $T_{0,1}^*$ for $u\in C^{\infty}(M;\mathbb{C})$, 
we introduce the Kohn-Spencer laplacian $\square_b$ defined by
\[
 \langle \square_b u,v \rangle_\theta
 = \langle \overline{\partial}_b u,
    \overline{\partial}_b v \rangle_\theta, 
 \quad v\in C_0^\infty(M;\mathbb{C}).
\]
These two operator are related to each other by 
\[
 \square_b = \Delta_b + \kyosu nT \quad\text{on $C^{\infty}(M;\mathbb{C})$.}
\]

\subsection{Tanaka-Webster connection}
We now review a connection due to Tanaka \cite{tanaka} and 
Webster \cite{webster}. 

Let $J\colon H\to H$ be the complex structure related to $(M,T_{1,0})$; 
that is, 
the $\mathbb{C}$-linear extension of $J$ is the 
multiplication by $\kyosu$ on $T_{1,0}$ and $-\kyosu$ on $T_{0,1}$, 
where we have used the fact that $H\otimes_{\mathbb{R}}\mathbb{C}=T_{1,0}\oplus T_{0,1}$.  
Moreover, we extend $J$ linearly to $TM$ by $J(T)=0$. 

Since $TM=H\oplus\mathbb{R}T=\{X+aT\mid X\in H, a\in\mathbb{R}\}$, there exists the unique
Riemannian metric $g_{\theta}$ on $M$ satisfying that
\[
g_{\theta}(X, Y)=d\theta(X,JY), \quad g_{\theta}(X,T)=0, \quad g_{\theta}(T,T)=1
\]
for $X, Y\in H$. 
$g_{\theta}$ is called the Webster metric.
We extend $g_{\theta}$ to $\mathbb{C}TM$ $\mathbb{C}$-bilinearly. 

The Tanaka-Webster connection is the unique linear connection $\nabla$ on $M$
satisfying that 
\begin{align}
&\nabla_X Y\in\Gamma^{\infty}(H), \quad X\in\Gamma^{\infty}(TM), Y\in\Gamma^{\infty}(H), \label{eqn.nabla.parallel}\\
&\nabla J=0, \quad \nabla g_{\theta}=0, \label{eqn.nabla.compatible}\\
& T_{\nabla}(Z,W)=0, \quad Z,W\in \Gamma^{\infty}(T_{1,0}), \\
& T_{\nabla}(Z,W)=2\kyosu L_{\theta}(Z,W)T, \quad Z\in \Gamma^{\infty}(T_{1,0}), W\in \Gamma^{\infty}(T_{0,1}),\\
& T_{\nabla}(T,J(X))+J(T_{\nabla}(T,X))=0, \quad X\in \Gamma^{\infty}(TM),
\end{align}
where $\nabla_X$ is the covariant derivative in the direction of $X$ and 
$T_{\nabla}$ is the torsion tensor field of $\nabla$: 
$T_{\nabla}(Z,W)=\nabla_{Z}W-\nabla_{W}Z-[Z,W]$. 

Let $\{Z_{\alpha}\}_{\alpha\in\langle n\rangle}$, where $\langle n\rangle=\{1,\ldots,n\}$, be a local
orthonormal frame for $T_{1,0}$ on an open set $U$, that is, 
$Z_{\alpha}$ is a $T_{1,0}$-valued section defined on $U$ and 
$g_{\theta}(Z_{\alpha},Z_{\overline{\beta}})=\delta_{\alpha\beta}$, where 
$Z_{\overline{\beta}}=\overline{Z_{\beta}}$.  
If we set $\llangle n\rrangle=\{0,1,\ldots,n,\overline{1},\ldots,\overline{n}\}$ and $Z_0=T$, 
then $\{Z_A\}_{A\in\llangle n\rrangle}$ is a local frame for $\mathbb{C}TM$. 
We define Christoffel symbols 
$\Gamma_{AB}^C$ for 
$A, B, C\in\llangle n\rrangle$ by
\[
\nabla_{Z_A}Z_B=\sum_{C\in\llangle n\rrangle}\Gamma_{AB}^C Z_C.
\]
Note that $\Gamma_{AB}^C=0$ unless 
$(B,C)\in\{(\beta,\gamma),(\overline{\beta},\overline{\gamma});\beta,\gamma\in\langle n\rangle\}$, 
because $\nabla_X(\Gamma^{\infty}(T_{1,0}))\subset \Gamma^{\infty}(T_{1,0})$ and 
$\nabla T=0$ by the conditions (\ref{eqn.nabla.parallel}) and (\ref{eqn.nabla.compatible}). 
We also have that
\begin{equation}\label{eqn.nabla.antisym}
\Gamma_{A\beta}^{\gamma}+\Gamma_{A\overline{\gamma}}^{\overline{\beta}}=0,\quad
\beta,\gamma\in\langle n\rangle, A\in\llangle n \rrangle
\end{equation}
by the condition (\ref{eqn.nabla.compatible}).

\subsection{Canonical vector fields}
To construct diffusion processes on CR manifolds in the next section, 
we introduce suitable vector bundles and principal bundles on them. 
The method employed there is a modification of 
the Eells-Elworthy-Malliavin method, 
the one to construct the 
Brownian motion on a Riemannian manifold 
via the SDE on the orthonormal frame bundle 
\cite{elworthy, hsu, i-w, malliavin, stroock}. 

Let $p\colon [a,b]\to M$, where $a<b$, be a smooth curve.
We say that a smooth curve $W\colon[a,b]\to T_{1,0}$ is 
a parallel section along $p$ if
$W(t)\in(T_{1,0})_{p(t)}$ and $\nabla_{\dot{p}}W=0$,  
where the dot always means the differentiation in $t$. 

For any $v\in (T_{1,0})_{p(a)}$, there exists a unique parallel section 
$W$ with $W(a)=v$. 
This can be seen by the localization argument as follows:
Let $\{Z_{\alpha}\}_{\alpha\in\langle n\rangle}$
be a local orthonormal frame for $T_{1,0}$ on $U$ and 
suppose that $p([a,b])\subset U$.
Then for a smooth curve $W\colon[a,b]\to T_{1,0}$ satisfying 
$W(t)\in(T_{1,0})_{p(t)}$, 
it holds that 
\[
W(t)=\sum_{\alpha\in\langle n\rangle}c^{\alpha}(t)(Z_{\alpha})_{p(t)}, 
\]
where $c^{\alpha}(t)=g_{\theta}(W(t),(Z_{\overline{\alpha}})_{p(t)})$ for 
$\alpha\in\langle n\rangle$. 
By the very definition of the covariant derivative, 
\begin{align*}
\nabla_{\dot{p}}W(t)
&= \sum_{\alpha\in\langle n\rangle}(\dot{c}^{\alpha}(t)(Z_{\alpha})_{p(t)}
+c^{\alpha}(t)\nabla_{\dot{p}}Z_{\alpha}(t))\\
&= \sum_{\alpha\in\langle n\rangle}\dot{c}^{\alpha}(t)(Z_{\alpha})_{p(t)}
+\sum_{\substack{A\in\llangle n\rrangle\\ \alpha,\beta\in\langle n\rangle}}
c^{\alpha}(t)g_{\theta}(\dot{p}(t),(Z_{\overline{A}})_{p(t)})
\Gamma_{A\alpha}^{\beta}(p(t))(Z_{\beta})_{p(t)},
\end{align*}
where
we have used the convention that $\overline{0}=0$. 
Therefore $\nabla_{\dot{p}}W=0$ if and only if
\[
\dot{c}^{\beta}(t)+\sum_{\substack{A\in\llangle n\rrangle\\ \alpha\in\langle n\rangle}}
c^{\alpha}(t)g_{\theta}(\dot{p}(t),(Z_{\overline{A}})_{p(t)})
\Gamma_{A\alpha}^{\beta}(p(t))=0
\]
for each $\beta\in\langle n\rangle$. 
Now, as an elementary application of the theory of ordinary differential equations, 
given $v\in (T_{1,0})_{p(a)}$ there exists a unique parallel section $W$ along $p$ such that
$W(a)=v$. 

Parallel sections can be represented locally as follows:

\begin{lemma}\label{lem.Up}
Let $\{Z_{\alpha}\}_{\alpha\in\langle n\rangle}$
be a local orthonormal frame for $T_{1,0}$ on $U$ and 
suppose that $p([a,b])\subset U$.
Then there exists a unique $\Lambda_p\colon [a,b]\to U(n)$,
where $U(n)$ is the group of $n\times n$ unitary matrices, 
such that $\Lambda_p(a)=I_n$, the identity matrix, and
\begin{equation}\label{eqn.Up}
\dot{\Lambda}_p(t)_{\beta}^{\gamma}+
\sum_{\substack{A\in\llangle n\rrangle\\ \delta\in\langle n\rangle}}\Lambda_p(t)_{\beta}^{\delta}g_{\theta}(\dot{p}(t),(Z_{\overline{A}})_{p(t)})\Gamma_{A\delta}^{\gamma}(p(t))=0, 
\end{equation}
where $\dot{\Lambda}_p(t)=(\dot{\Lambda}_p(t)_{\beta}^{\gamma})_{\gamma,\beta\in\langle n\rangle}$, 
holds for each $\beta, \gamma\in\langle n\rangle$. 
Moreover, given $v\in (T_{1,0})_{p(a)}$,
\[
W(t)=\sum_{\beta,\gamma\in\langle n\rangle}\Lambda_p(t)_{\beta}^{\gamma}g_{\theta}(v,(Z_{\overline{\beta}})_{p(a)})(Z_{\gamma})_{p(t)}
\]
is a parallel section along $p$ and satisfies $W(a)=v$. 
\end{lemma}
\begin{proof}
It is clear that the condition for $\Lambda_p$ defines a unique curve on 
$M_n(\mathbb{C})$, the group of $n\times n$ complex matrices.
By (\ref{eqn.nabla.antisym}) and (\ref{eqn.Up}) it is easy to check that 
\[
\frac{d}{dt}\biggl(\sum_{\gamma\in\langle n\rangle}
\Lambda_p(t)_{\alpha}^{\gamma}\overline{\Lambda_p(t)_{\beta}^{\gamma}}\biggr)=0
\]
holds for $\alpha,\beta\in\langle n\rangle$, 
which in conjunction with $\Lambda_p(a)=I_n$ implies that $\Lambda_p(t)\in U(n)$. 

We next show the second assertion.
Recall that  
\[
\nabla_{\dot{p}}W(t)
=\sum_{\beta,\gamma\in\langle n\rangle}g_{\theta}(v,(Z_{\overline{\beta}})_{p(a)})
(
\dot{\Lambda}_p(t)_{\beta}^{\gamma}(Z_{\gamma})_{p(t)}+\Lambda_p(t)_{\beta}^{\gamma}\nabla_{\dot{p}}Z_{\gamma}(t)
).
\]
Plugging (\ref{eqn.Up}) and the identity 
\[
\nabla_{\dot{p}}Z_{\gamma}(t)
=\sum_{A\in\llangle n\rrangle}g_{\theta}(\dot{p}(t),(Z_{\overline{A}})_{p(t)})
(\nabla_{Z_A}Z_{\gamma})_{p(t)} 
\]
into this, we obtain the desired equality $\nabla_{\dot{p}}W=0$. 
\end{proof}

Now we introduce the bundles over $M$ given by
\begin{align*}
L(T_{1,0})&=\coprod_{x\in M}\{r\colon\mathbb{C}^n\to(T_{1,0})_x; 
\text{$r$ is a non-singular linear map}\}, \\
U(T_{1,0})&=\{r\in L(T_{1,0}); \text{$r$ is isometric}\}. 
\end{align*}
For $r\in L(T_{1,0})$ with $r\colon \mathbb{C}^n\to(T_{1,0})_x$, let $\pi(r)=x$. 
We write $r\xi$ for the image of $\xi\in\mathbb{C}^n$ by $r\in L(T_{1,0})$. 

The Lie group $U(n)$ acts on $U(T_{1,0})$; 
for each $\Lambda\in U(n)$ we have the map
$R_\Lambda\colon U(T_{1,0})\to U(T_{1,0})$ defined by 
\[
(R_\Lambda r)(\xi)=r\Lambda \xi, \quad r\in U(T_{1,0}), \xi\in\mathbb{C}^n. 
\]
Moreover, if 
$\Lambda:[a,b]\to U(n)$ is a smooth curve with $\Lambda(a)=I_n$ and $r\in U(T_{1,0})$, 
then $\dot{\Lambda}(a)$ is a skew Hermitian matrix and
$\frac{d}{dt}\!\Bigm|_{t=a}R_{\Lambda(t)}r=\lambda(\dot{\Lambda}(a))_r$, where $\lambda$ is given by
\[
\lambda(u)_r=\frac{d}{ds}\!\biggm|_{s=0}\!R_{\exp(su)}r. 
\]
It should be remarked that, while the fiber of $U(T_{1,0})$ is a complex vector space and 
complex group $U(n)$ acts on it, $U(T_{1,0})$ is a real manifold since 
so is the base manifold $M$. 

For smooth curves $p\colon [a,b]\to M$ and $\widehat{p}\colon [a,b]\to U(T_{1,0})$, 
we say that $\widehat{p}$ is a horizontal lift of $p$ to $U(T_{1,0})$ if 
$\pi\circ \widehat{p}=p$ and $\widehat{p}(t)e_{\alpha}$ 
is a parallel section for any $\alpha\in\langle n\rangle$, where
$\{e_{\alpha}\}_{\alpha\in\langle n\rangle}$ is the standard coordinate of $\mathbb{C}^n$. 
For $v\in T_xM$, $r\in\pi^{-1}(x)$ and $\eta\in T_rU(T_{1,0})$, 
we say that $\eta$ is a horizontal lift of $v$ if 
there exist a smooth curve $p$ on $M$ and 
a smooth curve $\widehat{p}$ on $U(T_{1,0})$ which is a horizontal lift of $p$, satisfying
$\widehat{p}(0)=r$, $\dot{\widehat{p}}(0)=\eta$ and $\pi_* \eta=v$. 

For each $v\in T_xM$, there exists a unique horizontal lift
$\eta_x(r)\in T_rU(T_{1,0})$. 
Namely, let $\{Z_{\alpha}\}_{\alpha\in\langle n\rangle}$ be a
local orthonormal frame for $T_{1,0}$ on $U$ and 
suppose the curve $p$ is contained in $U$.  
Let $Z\colon U\to U(T_{1,0})$ be the section determined by
$\{Z_{\alpha}\}_{\alpha\in\langle n\rangle}$,  
i.e. $Z(x)e_{\alpha}=(Z_{\alpha})_x$. 
By virtue of Lemma \ref{lem.Up}, 
$\widehat{p}\colon[a,b]\to U(T_{1,0})$ is a horizontal lift of $p$ if and only if 
$\pi(\widehat{p}(a))=p(a)$ and 
\begin{equation}\label{eqn.horlift}
\widehat{p}(t)=R_{Z(p(a))^{-1}\circ \widehat{p}(a)}\circ R_{\Lambda_p(t)}Z(p(t)), 
\end{equation}
where $Z(p(a))^{-1}\circ \widehat{p}(a)\colon\mathbb{C}^n\to\mathbb{C}^n$ is 
regarded as an element of $U(n)$ and
$\Lambda_p$ is the curve defined in 
Lemma \ref{lem.Up}. 

Under the identification of $Z(x)\in U(T_{1,0})$ with
$((Z_1)_x,\ldots,(Z_n)_x)\in(T_{1,0})_x^n$, 
we have $R_\Lambda Z(x)=Z(x)\Lambda$ for $\Lambda\in U(n)$. 
Then we can calculate as
\[
\frac{d}{dt}(R_{\Lambda_p(t)}Z(p(t)))
=\frac{d}{dt}(Z(p(t))\Lambda_p(t))
=Z_*(\dot{p}(t))\Lambda_p(t)+Z(p(t))\dot{\Lambda}_p(t).
\]
By differentiating (\ref{eqn.horlift}) and 
substituting the above identity, we arrive at 
the unique horizontal lift of $v$:
\begin{equation}\label{eqn.horliftvec}
\eta_r(v)=(R_{Z(x)^{-1}\circ r})_* (Z_*(v)-\lambda(\Phi(v))_{Z(x)})
\in T_rU(T_{1,0}), 
\end{equation}
where $\Phi\colon T_xM\to \mathfrak{u}(n)$, 
$\mathfrak{u}(n)$ being the set of $n\times n$ skew Hermitian matrices, 
is defined by
\[
\Phi(v)
=\biggl(
\sum_{A\in\llangle n\rrangle}g_{\theta}(v,(Z_{\overline{A}})_x)\Gamma_{A\beta}^{\gamma}(x)
\biggr)_{\beta,\gamma\in\langle n\rangle}.
\]

For each $r\in U(T_{1,0})$, 
the horizontal subspace at $r$ is defined by 
\[
\mathrm{Hor}_rU(T_{1,0})=\{\eta_r(v);v\in T_xM\}\subset T_rU(T_{1,0}). 
\]
If we set $\mathrm{Ver}_rU(T_{1,0})=\mathrm{Ker}(\pi_*\colon T_rU(T_{1,0})\to T_{\pi(r)}M)$, 
the vertical subspace, then
\[
T_rU(T_{1,0})=\mathrm{Ver}_rU(T_{1,0})\oplus\mathrm{Hor}_rU(T_{1,0})
\]
holds. 

$\eta_r$ extends naturally to a $\mathbb{C}$-linear map from 
$T_xM\otimes_{\mathbb{R}}\mathbb{C}$ to $\mathrm{Hor}_rU(T_{1,0})\otimes_{\mathbb{R}}\mathbb{C}$. 
Then for each $\xi\in\mathbb{C}^n$, we can define the canonical vector field $L(\xi)$ by
$L(\xi)_r=\eta_r(r\xi)$ for $r\in U(T_{1,0})$. 
We set
\[
L_{\alpha}=L(e_{\alpha}), \quad \alpha\in\langle n\rangle
\]
and call $\{L_{\alpha}\}_{\alpha\in\langle n\rangle}$ 
the canonical vector fields.  

Let $\{Z_{\alpha}\}_{\alpha\in\langle n\rangle}$ be a local orthonormal frame for $T_{1,0}$ on $U$. 
Define $\{e_{\alpha}^{\beta}(r)\}\in\mathbb{C}^n\otimes\mathbb{C}^n$ for $r\in L(T_{1,0})$ with $\pi(r)\in U$ by
$r(e_{\alpha})=\sum_{\beta\in\langle n\rangle}e_{\alpha}^{\beta}(r)(Z_{\beta})_{\pi(r)}$. 
We can then introduce a local coordinate system $\{(x^k,e_{\alpha}^{\beta})\}$ of $L(T_{1,0})$, 
$(x^k)_{1\le k\le 2n+1}$ being a local coordinate system of $M$. 
With respect to this coordinate 
we represent the canonical vector field $L_{\alpha}$, $\alpha\in\langle n\rangle$ as follows. 

Recall that $U(T_{1,0})$ can be identified with $M\times U(n)$ locally, 
and under this identification $R_\Lambda((x,e))=(x,e\Lambda)$ for $(x,e)\in M\times U(n)$. 
Therefore it holds that
\[
(R_{Z(x)^{-1}\circ r})_*Z_*(v)
=v, \quad v\in T_xM\otimes_{\mathbb{R}}\mathbb{C},
\]
where $Z\colon U\to U(T_{1,0})$ is the section determined by
$\{Z_{\alpha}\}_{\alpha\in\langle n\rangle}$ as before. 
Since
\begin{align*}
\lambda\biggl(\Phi\biggl(
\mathrm{Re}\sum_{\beta}e_{\alpha}^{\beta}Z_{\beta}\biggr)\biggr)
&=
\frac{1}{2}
\sum_{\beta,\gamma,\delta\in\langle n\rangle}
\biggl(
(
e_{\alpha}^{\beta}\Gamma_{\beta\delta}^{\gamma}
+e_{\overline{\alpha}}^{\overline{\beta}}\Gamma_{\overline{\beta}\delta}^{\gamma}
)\frac{\partial}{\partial e_{\delta}^{\gamma}}
+
(
e_{\overline{\alpha}}^{\overline{\beta}}\Gamma_{\overline{\beta}\overline{\delta}}^{\overline{\gamma}}
+e_{\alpha}^{\beta}\Gamma_{\beta\overline{\delta}}^{\overline{\gamma}}
)\frac{\partial}{\partial e_{\overline{\delta}}^{\overline{\gamma}}}
\biggr)
, \\
\lambda\biggl(\Phi\biggl(
\mathrm{Im}\sum_{\beta}e_{\alpha}^{\beta}Z_{\beta}\biggr)\biggr)
&=\frac{1}{2\kyosu}
\sum_{\beta,\gamma,\delta\in\langle n\rangle}
\biggl(
(
e_{\alpha}^{\beta}\Gamma_{\beta\delta}^{\gamma}
-e_{\overline{\alpha}}^{\overline{\beta}}\Gamma_{\overline{\beta}\delta}^{\gamma}
)\frac{\partial}{\partial e_{\delta}^{\gamma}}
-
(
e_{\overline{\alpha}}^{\overline{\beta}}\Gamma_{\overline{\beta}\overline{\delta}}^{\overline{\gamma}}
-e_{\alpha}^{\beta}\Gamma_{\beta\overline{\delta}}^{\overline{\gamma}}
)\frac{\partial}{\partial e_{\overline{\delta}}^{\overline{\gamma}}}
\biggr)
,  
\end{align*}
we have from (\ref{eqn.horliftvec}) that
\begin{equation}\label{eqn.canvec}
(L_{\alpha})_r
=
\sum_{\beta\in\langle n\rangle}e_{\alpha}^{\beta}Z_{\beta}
-\sum_{\beta,\gamma,\delta,\varepsilon\in\langle n\rangle}
\Gamma_{\beta\delta}^{\gamma}
e_{\varepsilon}^{\delta}
e_{\alpha}^{\beta}
\frac{\partial}{\partial e_{\varepsilon}^{\gamma}}
-\sum_{\beta,\gamma,\delta,\varepsilon\in\langle n\rangle}
\Gamma_{\beta\overline{\delta}}^{\overline{\gamma}}
e_{\overline{\varepsilon}}^{\overline{\delta}}
e_{\alpha}^{\beta}
\frac{\partial}{\partial e_{\overline{\varepsilon}}^{\overline{\gamma}}}. 
\end{equation}

\section{Construction of a diffusion process}\label{sec.diffusion}
In this section, we construct a diffusion process 
$\mathbb{X}=\{(\{X(t)\}_{t\ge0},P_x);x\in M\}$ generated by $-\Delta_b/2$.

Let $\{L_\alpha\}_{\alpha\in\langle n\rangle}$ 
be the canonical vector fields on $U(T_{1,0})$ constructed
in the previous section.  
Take a $\mathbb{C}^n$-valued
Brownian motion 
$\{B(t)=(B^1(t),\dots,B^n(t))\}_{t\ge0}$, that is, 
$\{B(t)\}_{t\ge0}$ is a
$\mathbb{C}^n$-valued continuous 
martingale with $\langle B^\alpha,B^\beta\rangle(t)=0$ and 
$\langle B^\alpha,B^{\overline{\beta}}\rangle(t)= \delta_{\alpha \beta} t$, 
where $\langle M,N\rangle(t)$ denotes the quadratic
variation of continuous martingales  
$\{M(t)\}_{t\ge0}$ and $\{N(t)\}_{t\ge0}$. 
Let $\{r(t)=r(t,r,B)\}_{t\ge0}$ be the
unique solution to an 
SDE on $U(T_{1,0})$:  
\begin{equation}\label{eqn.3.1}
 dr(t) =\sum_{\alpha\in\langle n\rangle}
(L_\alpha(r(t))\circ dB^\alpha(t) +
 L_{\overline{\alpha}}(r(t)) \circ dB^{\overline{\alpha}}(t)), 
 \quad r(0) =r \in U(T_{1,0}),
\end{equation}
or equivalently
\begin{equation}\label{eqn.3.2}
 dr(t) =\sum_{\alpha\in\langle n\rangle}(
\sqrt{2}\, \Re L_\alpha(r(t)) \circ d\xi^\alpha(t)
+\sqrt{2}\, \Im L_\alpha(r(t)) \circ d\eta^\alpha(t)),
 \quad r(0) =r \in U(T_{1,0}),
\end{equation}
where $\Re L_\alpha=(L_\alpha+L_{\overline\alpha})/2$, 
$\Im L_\alpha=(L_\alpha-L_{\overline\alpha})/2\kyosu$, 
$\xi^\alpha(t)=\sqrt{2}\, \Re B^\alpha(t)$ and
$\eta^\alpha(t)=\sqrt{2}\, \Im B^\alpha(t)$.  
The process $r(t)$ may explode.  
Note that 
$(\xi^1(t),\eta^1(t),\dots,\xi^n(t), \eta^n(t))$ is an $\mathbb{R}^{2n}$-valued Brownian motion.
The manipulation of taking the real part on the right hand side of 
the SDE (\ref{eqn.3.1}) is due to that $U(T_{1,0})$ is a real manifold. 

Let $\{Z_\alpha\}_{\alpha\in\langle n\rangle}
$ be a local orthonormal frame for $T_{1,0}$  
and $(x^k,e_\alpha^\beta)$ be the associated
local coordinate of $L(T_{1,0})$ as in the previous section.  
Then (\ref{eqn.3.1}) can be rewritten locally as
\begin{equation}\label{eqn.3.3}
 \left\{
 \begin{aligned}
   dx(t) & = \sum_{\alpha,\beta\in\langle n\rangle}(e_\alpha^\beta(t) Z_\beta(x(t))\circ dB^\alpha(t)
           +e_{\overline{\alpha}}^{\overline{\beta}}(t)
             Z_{\overline{\beta}}(x(t)) \circ  
             dB^{\overline{\alpha}}(t)),
   \\
   de_{\varepsilon}^{\gamma}(t) & = 
     - \sum_{\alpha,\beta,\delta\in\langle n\rangle}
(\Gamma_{\beta\delta}^{\gamma}(x(t)) e_{\varepsilon}^{\delta}(t) e_{\alpha}^{\beta}(t)
\circ dB^{\alpha}(t)
     + 
\Gamma_{\overline{\beta}\delta}^{\gamma}(x(t)) e_{\varepsilon}^{\delta}(t) e_{\overline{\alpha}}^{\overline{\beta}}(t)
\circ dB^{\overline{\alpha}}(t)).
 \end{aligned}
 \right.
\end{equation}
Hence it follows from the uniqueness of
$\{r(t,r,B)\}_{t\ge0}$ that 
$r(t,r\Lambda,\overline{\Lambda}B) 
 = r(t,r,B)$ for every unitary matrix $\Lambda$.  
Denoting by $\widetilde M$ a one-point
compactification of $M$, 
we have that the induced measures $Q_r$ of
$\pi(r(\cdot,r,B))$ on $C([0,\infty);\widetilde M)$, the space of
$\widetilde M$-valued continuous functions defined on $[0,\infty)$,
coincide for all $r\in \pi^{-1}(x)$.  Put
\[
 P_x = Q_r,\quad r\in \pi^{-1}(x).
\]

Set 
\[
 \mathcal{L} = \frac{1}{2} \sum_{\alpha\in\langle n\rangle} 
(L_\alpha L_{\overline{\alpha}} + L_{\overline{\alpha}} L_\alpha) |_M. 
\]
It is easily seen that 
\[
f(X(t))-\int_0^t {\mathcal L} f(X(s)) ds
\]
is a martingale under $P_x$ for every $x\in M$ and 
$f\in C_0^\infty(M)$, where $X(t)$ denotes the position of
$X\in C([0,\infty);\widetilde M)$ at time $t$. 
By a straightforward computation, we have a local representation of 
$\mathcal{L}$ as follows:
\begin{equation}\label{eqn.3.5}
 \mathcal{L} = \frac{1}{2} \biggl(\sum_{\alpha\in\langle n\rangle}
(Z_\alpha Z_{\overline{\alpha}} + Z_{\overline{\alpha}}Z_\alpha)
-
 \sum_{\alpha,\beta\in\langle n\rangle}  
 (\Gamma_{\overline{\beta}\beta}^\alpha Z_\alpha 
 + \Gamma_{\beta\overline{\beta}}^{\overline{\alpha}} Z_{\overline{\alpha}})
\biggr).
\end{equation}
Recall, moreover, an identity that 
\[
d_b f = \sum_{\alpha\in\langle n\rangle}(Z_\alpha f \theta^\alpha + Z_{\overline{\alpha}} f \theta^{\overline{\alpha}})
\]
and Greenleaf's result \cite{greenleaf} that
\[
 \langle Z_\alpha f,\overline{g}\rangle_\theta
 = \left\langle f,\overline{\bigl(-Z_{\overline{\alpha}}+\sum_\beta
 \Gamma_{\beta\overline{\beta}}^\alpha \bigr)g} 
 \right\rangle_\theta.
\]
Plugging these into (\ref{eqn.3.5}), we see that
\[
 \mathcal{L} = -\frac{1}{2} \Delta_b.
\]
Thus we have shown that
\begin{theorem}\label{thm.diffusion}
There exists a diffusion process $\mathbb{X}=\{(\{X(t)\}_{t\ge0},P_x);x\in M\}$
generated by $-\Delta_b/2$ and which is obtained via the SDE
\eqref{eqn.3.1}.
\end{theorem}

\begin{example}
Let $\mathbb{H}_n=\mathbb{C}^n\times \mathbb{R}$ be the
$(2n+1)$-dimensional Heisenberg group with a coordinate system
$(z,t)$, $z=(z^1,\ldots,z^n)\in \mathbb{C}^n$, $t\in \mathbb{R}$.  Define 
\begin{align*}
 & \theta = \frac{1}{2} 
\biggl(dt - \kyosu\sum_{\alpha=1}^n 
( 
 \overline{z^{\alpha}} dz^\alpha - z^\alpha d\overline{z^{\alpha}}
) \biggr),
 \\
 &
 T_{1,0} = \bigoplus_{\alpha=1}^n \mathbb{C}Z_{\alpha}
, \quad\text{where }
 Z_\alpha = \frac{\partial}{\partial z^\alpha} + \kyosu
 \overline{z^{\alpha}} \frac{\partial}{\partial t}.
\end{align*}
Then $\mathbb{H}_n$ is a strictly pseudoconvex CR manifold, 
see \cite{dragomir-tomassini}.  Since
$\{Z_\alpha\}_{\alpha\in\langle n\rangle}$ is a global orthonormal frame for $T_{1,0}$ and 
\[
 d\theta = \kyosu\sum_{\alpha=1}^n dz^\alpha\wedge
 d\overline{z^{\alpha}}, 
 \quad d(dz^\alpha)=0,
\]
the associated covariant derivation is a null mapping. 
In particular, 
$\Delta_b=-\sum_\alpha \bigl(Z_\alpha Z_{\overline{\alpha}} + Z_{\overline{\alpha}}Z_\alpha\bigr)$. 
The diffusion process described in Theorem \ref{thm.diffusion} is
exactly the same one as that studied by Gaveau in \cite{gaveau}.
\end{example}

\begin{remark}
Diffusion processes on sub-Riemannian manifolds, which include CR manifolds, 
are studied from the point of view of sub-Riemannian geometry. 
For example, In Gordina-Laetsch \cite{gordina-laetsch} diffusion processes
on sub-Riemannian manifolds are constructed as the limit of random walks
constructed piecewisely via the Hamiltonian-flow associated with
a sub-Riemannian structure. 
\end{remark}

\section{Heat kernel and stochastic line integral}\label{sec.heat}
In this section, we apply the result \cite{taniguchi-ZW} on
partial hypoellipticity to the diffusion process constructed
in the previous section and stochastic line integrals along
the diffusion process.

We first consider the heat
equation 
\begin{equation}\label{eqn.heat.eq}
\frac{\partial}{\partial t}u=-\frac{1}{2}\Delta_b u, \quad
u(0,x)=f(x), f\in C_b^{\infty}(M), 
\end{equation}
via the diffusion process $\mathbb{X}=\{(\{X(t)\}_{t\ge0},P_x);x\in M\}$ constructed
in Theorem \ref{thm.diffusion}.

By Whitney's embedding theorem, we may think of $U(T_{1,0})$ as a
closed submanifold of $\mathbb{R}^k$ for some $k$.  We further
assume that

\begin{itemize}
\item[(H)] there exist $C^\infty$ vector fields $L'_\alpha$,
$\alpha\in\langle n\rangle$, on $\mathbb{R}^k$ with $\mathbb{C}$-valued
coefficients such that 
(i) $L_\alpha=L'_\alpha$ on $U(T_{0,1})$, and 
(ii) the coefficients of $L'_\alpha$ and their derivatives of all orders are bounded.
\end{itemize}

\smallskip\noindent
The hypothesis (H) implies that $r(t)$ does not explode. 
For example, this hypothesis is fulfilled if $M$ is compact.  We shall establish

\begin{theorem}\label{thm.density}
Assume that (H) holds.  Then there is a $p\in
C^\infty((0,\infty)\times M\times M)$ such that
\[
 P_x(X(t)\in dy) = p(t,x,y)\psi(dy).
\]
\end{theorem}

\begin{proof}
Recall the expression (\ref{eqn.3.2}). 
By virtue of \cite[Theorem 3.1]{taniguchi-ZW} and 
\cite[Lemma 3.1]{taniguchi-OJM}, it suffices to show that
\begin{equation}\label{eqn.3.6}
 \mathrm{span}_\mathbb{R} \{
 (\pi_*)_r \Re L_\alpha, (\pi_*)_r \Im L_\alpha,
 (\pi_*)_r [\Re L_\alpha,\Im L_\alpha]; \alpha\in\langle n\rangle
 \} = T_{\pi(r)} M
\end{equation}
for every $r\in U(T_{1,0})$,
where $\mathrm{span}_\mathbb{R}$ stands for taking all real linear
combinations. 

To see this, let
$\{Z_\alpha\}_{\alpha\in\langle  n\rangle}$ be a local
orthonormal frame for $T_{1,0}$, and 
$\{\theta,\theta^\alpha,\theta^{\overline{\alpha}}\}$ be the dual basis of
$\{T,Z_\alpha,Z_{\overline{\alpha}}\}$. 
By (\ref{eqn.canvec}), it holds that 
\begin{equation}\label{eqn.3.7}
 (\pi_*) L_\alpha = \sum_{\beta\in\langle n\rangle}e_\alpha^\beta Z_\beta.
\end{equation}

We next observe that
\begin{equation}\label{eqn.3.8}
 (\pi_*)[L_\alpha,L_{\overline{\alpha}}] = -2\kyosu T
 \quad\bmod\{Z_\alpha,Z_{\overline{\alpha}}\}_{\alpha},
\end{equation}
where we have meant by 
``$A=B$ mod $\{Z_\alpha,Z_{\overline{\alpha}}\}_{\alpha}$'' that 
$A=B+\sum_{\alpha\in\langle n\rangle}a^\alpha Z_\alpha + \sum_{\alpha\in\langle n\rangle}b^{\overline{\alpha}}
Z_{\overline{\alpha}}$ for some 
$a^\alpha,b^{\overline{\alpha}}\in \mathbb{C}$. 
For this purpose, recall that 
\[
 d\theta(Z,W) = \frac{1}{2} ( Z(\theta(W))-W(\theta(Z))
 -\theta([Z,W]) ).
\]
Since $\theta(T_{1,0}\oplus T_{0,1})=0$, it holds that
\begin{equation}\label{eqn.3.9}
 \theta([Z_\alpha,Z_{\overline{\beta}}])
= -2d\theta(Z_\alpha,Z_{\overline{\beta}})
= -2\kyosu L_\theta( Z_\alpha,Z_{\overline{\beta}})
= -2\kyosu\delta_{\alpha \beta}. 
\end{equation}
Hence 
\begin{equation}\label{eqn.3.10}
 [Z_\alpha,Z_{\overline{\beta}}] = -2\kyosu\delta_{\alpha\beta} T
 \quad
 \mathrm{mod}\ \{Z_\alpha,Z_{\overline{\alpha}}\}_{\alpha},
\end{equation}
which yields that (\ref{eqn.3.8}) holds.

(\ref{eqn.3.6}) follows from (\ref{eqn.3.7}) and (\ref{eqn.3.8}).
\end{proof}

\begin{remark}
By Theorem \ref{thm.density}, 
a bounded solution to the heat equation (\ref{eqn.heat.eq}) can be written as 
\[
u(t,x)=E_x[f(X(t))]=\int_{M}f(y)p(t,x,y)\psi(dy).
\] 
\end{remark}

We next investigate stochastic line integrals.
Let $\Xi$ be a 1-form on $M$, which, under the imbedding
made in the assumption (H), can be extended to a 1-form on
$\mathbb{R}^k$ such that its derivatives of all orders are
bounded.

Denote by $\int_{X[0,t]}\Xi$ the stochastic line integral of
$\Xi$ along $\{X_t\}_{t\ge0}$ from time $0$ to $t$.
For definition, see \cite{i-w}.
It is easily checked that
\[
  \int_{X[0,t]}\Xi
  =\sum_{A\in\llangle n\rrangle\setminus\{0\}}
   \int_0^t (\pi^*\Xi)_{r(s)}(L_A)\circ dB^A(s),
\]
where $\pi^*\Xi$ is the pull-back of $\Xi$ through
$\pi\colon U(T_{1,0})\to M$ and $(\pi^*\Xi)_r(L_A)$ is the 
pairing of cotangent vector $(\pi^*\Xi)_r$ and tangent vector
$(L_A)_r$ at $r\in U(T_{1,0})$.
Thus, $\{\widetilde{r}(t)=(r(t),\int_{X[0,t]}\Xi)\}_{t\ge0}$
obeys the SDE
\[
    d\widetilde{r}(t)
     =\sum_{A\in\llangle n\rrangle\setminus\{0\}}
        \widetilde{L}_A(\widetilde{r}(t))\circ dB^A(t),
\]
where $\widetilde{L}_A$'s are vector fields on
$U(T_{1,0})\times \mathbb{R}$ defined by
\[
    \widetilde{L}_A
    =L_A+(\pi^*\Xi)(L_A)\frac{\partial}{\partial\xi},
\]
$\xi$ being the coordinate on $\mathbb{R}$.

For $x\in M$, take a local orthonormal frame
$\{Z_\alpha\}_{\alpha\in\langle n\rangle}$ for $T_{1,0}$ on
$U$, and set 
$\Xi_A=\Xi(Z_A)$ for $A\in\llangle n\rrangle\setminus\{0\}$.
For $A_1,\dots,A_m\in\llangle n\rrangle\setminus\{0\}$,
define $\Phi_{A_1,\dots,A_m}(\Xi)\colon U\to \mathbb{C}$
successively by
\[
    \Phi_{A_1}(\Xi)=\Xi_{A_1} \quad\text{and}\quad
    \Phi_{A_1,\dots,A_m}(\Xi)
    =Z_{A_1}\Phi_{A_2,\dots,A_m}(\Xi)
     -[Z_{A_2},[\dots,[Z_{A_{m-1}},Z_{A_m}]\dots]]\Xi_{A_1}.
\]

\begin{theorem}
Suppose that (H) holds and for each $x\in M$ there exists
$A_1,\dots,A_m\in\llangle n\rrangle\setminus\{0\}$ such that
$\Phi_{A_1,\dots,A_m}(\Xi)(x)\ne0$.
Then the distribution of $\int_{X[0,t]}\Xi$ under $P_x$
admits a smooth density function with respect to the Lebesgue
measure on $\mathbb{R}$ for every $x\in M$.
\end{theorem}

\begin{proof}
Under the same notation as used in \eqref{eqn.canvec},
set 
$(f_\alpha^\beta)_{\alpha,\beta\in\langle n\rangle}
 =(e_\alpha^\beta)_{\alpha,\beta\in\langle n\rangle}^{-1}$
and define locally
\[
    \widehat{L}_\alpha
    =\sum_{\beta\in\langle n\rangle} f_\alpha^\beta
       \widetilde{L}_\beta.
\]
Then it is easily seen that
\begin{align*}
   & \text{\rm span}_{\mathbb{C}}
     \bigl\{(\widetilde{L}_A)_r,
      ([\widetilde{L}_{A_1},[\dots,[\widetilde{L}_{A_{m-1}},
       \widetilde{L}_{A_m}]\dots]])_r;
      A,A_1,\dots,A_m\in\llangle n\rrangle\setminus\{0\},
      m=2,3,\dots\bigr\}
   \\
   & =\text{\rm span}_{\mathbb{C}}
     \bigl\{(\widehat{L}_A)_r,
      ([\widehat{L}_{A_1},[\dots,[\widehat{L}_{A_{m-1}},
       \widehat{L}_{A_m}]\dots]])_r;
      A,A_1,\dots,A_m\in\llangle n\rrangle\setminus\{0\},
      m=2,3,\dots\bigr\},
\end{align*}
and that
\[
    (\widetilde{\pi}_*)_r([\widehat{L}_{A_1},[\dots,[\widehat{L}_{A_{m-1}},
       \widehat{L}_{A_m}]\dots]]_r)
    =\Phi_{A_1,\dots,A_m}(\Xi)(\pi(r))
      \frac{\partial}{\partial\xi},
\]
where $\widetilde{\pi}\colon U(T_{1,0})\times\mathbb{R}\to\mathbb{R}$ is 
the natural projection.
Hence, applying \cite[Theorem 3.1]{taniguchi-ZW}, we obtain
the desired result.
\end{proof}

\begin{remark}
Although $Z_A$'s in the definition of $\Phi_{A_1,\ldots,A_m}(\Xi)$ are
all in $T_{1,0}\oplus T_{0,1}$, 
the direction $T$ appears in $\Phi_{A_1,\ldots,A_m}(\Xi)$'s because
the expression $[Z_{\alpha},Z_{\overline{\alpha}}]$ contains $T$-part by
(\ref{eqn.3.10}). 
Hence, for example, 
even if $\Xi_A(x)=0$ for each $A\in\llangle n\rrangle\setminus\{0\}$, 
the assumption $\Phi_{A_1,\ldots,A_m}(\Xi)(x)\ne0$ may be satisfied. 
\end{remark}

\section{Dirichlet problem}\label{sec.dirichlet}
In this section, we study Dirichlet problems related to $\Delta_b$.
For $f\in C(\partial G)$, what to be found is a 
$u_f\in C^2(G)\cap C(\overline{G})$ such that
$\Delta_b u_f=0$ and $u_f|_{\partial G}=f$. 
We first establish a weak solution in a probabilistic manner following 
Stroock and Varadhan \cite{stroock-varadhan}. 
As will be seen in Remark \ref{rem.hypoelliptic}, 
we indeed obtain a classical solution stated above. 

Let $\mathbb{X}=\{(\{X(t)\}_{t\ge0},P_x);x\in M\}$ be the diffusion process
obtained in Theorem \ref{thm.diffusion}. 
Let $G$ be a relatively
compact connected open set in $M$ with $C^3$ boundary.  Define
\[
 \tau^\prime = \inf\{t\ge0; X(t)\notin \overline G\}.
\]
We shall show that

\begin{theorem}\label{thm.dirichlet}
For $f\in C(\partial G)$, define 
$u_f(x)= E_x[f(X(\tau^\prime))]$. Then $u_f\in C(\overline{G})$ and
satisfies that 
\[
 \langle u_f,\Delta_b v\rangle_\theta = 0 \quad\text{for any }v\in
 C_0^\infty(G),  \quad\text{and}\quad u_f=f \quad\text{on }\partial
 G. 
\]
\end{theorem}
Due to the result by Stroock and Varadhan \cite{stroock-varadhan},
the theorem is verified once we have established the following two
lemmas.  

\begin{lemma}\label{lem.dirichlet.1}
It holds that
\[
 \sup_{x\in \overline{G}} E_x[\tau^\prime] < \infty.
\]
\end{lemma}

\begin{lemma}\label{lem.dirichlet.2}
Every boundary point is $\tau^\prime$-regular, that is,
\[
 P_x(\tau^\prime = 0) = 1, \quad x\in \partial G.
\]
\end{lemma}

\begin{proof}[Proof of Lemma \ref{lem.dirichlet.1}]
On account of \cite[Remark 5.2]{stroock-varadhan}, it suffices to
show that
\begin{equation}\label{eqn.4.5}
 P_x(\tau^\prime < T ) > 0, \quad x\in \overline{G}  
 \text{ and }T>0.
\end{equation}
To do this, take a family $\{U_j\}_{j=1}^N$ of coordinate
neighborhoods of $M$ such that $\overline{G}\subset \cup_{j=1}^N U_j$.  
Let $\Lambda = \{j;U_j\cap \partial G\ne\emptyset\}$. 
Take $j\in \Lambda$ and a local orthonormal frame
$\{Z_\alpha\}_{\alpha\in\langle n\rangle}$
for $T_{1,0}$ on $U_j$. 
Then, by virtue of \eqref{eqn.3.5}, we may assume that the part of
$\{X(t)\}_{t\ge0}$ on $U_j$ is governed by an SDE 
\begin{equation}\label{eqn.4.6}
 dX(t) = \sum_{\alpha\in\langle n\rangle}(\sqrt{2}\, \Re Z_\alpha(X(t)) \circ d\xi^\alpha(t) 
  + \sqrt{2}\, \Im Z_\alpha(X(t)) \circ d\eta^\alpha(t)) + b(X(t))
  dt, 
\end{equation}
where $(\xi^1(t),\eta^1(t),\dots,\xi^n(t),\eta^n(t))$ is an
$\mathbb{R}^{2n}$-valued Brownian motion and 
\[
 b=-\sum_{\alpha,\beta\in\langle n\rangle} (\Gamma_{\beta\overline{\beta}}^\alpha
 Z_\alpha + 
 \Gamma_{\overline{\beta} \beta}^{\overline{\alpha}}
 Z_{\overline{\alpha}}). 
\]
Due to (\ref{eqn.3.10}), applying the support theorem
(cf. \cite[Theorem 3.2]{kunita}), we obtain that
\begin{equation}\label{eqn.4.7}
 P_x(\tau^\prime < T) >0, \quad x\in U_j, j\in \Lambda, T>0.
\end{equation}

For $U_k$ such that $k\notin\Lambda$ and $U_k\cap U_j \ne \emptyset$
for some $j\in \Lambda$, by the same reasoning as above, applying
the support theorem again, we have that
\[
 P_x(\text{$X(t)$ hits $U_j$ before $T$}) >0,
 \quad x\in U_k, T>0.
\]
Combined with (\ref{eqn.4.7}) and the strong Markov property, this
yields that
\[
 P_x(\tau^\prime <T)>0, \quad x\in U_k, T>0. 
\]
Repeating this argument successively, we can conclude
\eqref{eqn.4.5}.
\end{proof}

\begin{proof}[Proof of Lemma \ref{lem.dirichlet.2}]
Let $x\in \partial G$ and $U$ be a coordinate neighborhood of $x$.
For a local orthonormal frame
$\{Z_\alpha\}_{\alpha\in\langle n\rangle}$ for
$T_{1,0}$ defined 
on $U$, we may and will assume that the part of $\{X(t)\}_{t\ge0}$ on $U$ obeys
the SDE \eqref{eqn.4.6}.

Let $\varphi$ be a local defining function of $G$ around $x$; there
is an open set $V$ containing $x$ such that $\varphi\in C^3(V)$,
$V\cap G=\{y\in V;\varphi(y)<0\}$, and $d\varphi(y)\ne0$ for $y\in \partial G \cap V$. 
If either $(\Re Z_\alpha)\varphi(x)\ne0$ or 
$(\Im Z_\alpha)\varphi(x)\ne0$, then by 
\cite[Corollary~4]{stroock-taniguchi}, $x$ is $\tau^\prime$-regular.
Now we suppose that 
\begin{equation}\label{eqn.4.8}
 (\Re Z_\alpha)\varphi(x)=(\Im Z_\alpha)\varphi(x)=0,
 \quad\alpha\in\langle n\rangle.
\end{equation}
Since $\{\Re Z_\alpha,\Im
Z_\alpha,T\}_{\alpha\in\langle n\rangle}$ forms a
local basis of $TM$ 
on $U$, this implies that $T\varphi(x) \ne 0$.  Moreover, in
conjunction with \eqref{eqn.3.10} and \eqref{eqn.4.8} also implies that
\[
 [\Re Z_\alpha,\Im Z_\alpha]\varphi(x) = T\varphi(x) \ne 0.
\]
Hence it follows that, for each $\alpha$, either $(\Re Z_\alpha)
(\Im Z_\alpha) \varphi(x)\ne0$ or $(\Im Z_\alpha)(\Re Z_\alpha)
\varphi(x)\ne0$ and that a matrix
\[
 \begin{pmatrix} (\Re Z_\alpha)(\Re Z_\beta)\varphi(x) &
 (\Re Z_\alpha)(\Im Z_\beta)\varphi(x) \\
 (\Im Z_\alpha)(\Re Z_\beta)\varphi(x) &
 (\Im Z_\beta)(\Im Z_\beta)\varphi(x) 
 \end{pmatrix}_{\alpha,\beta\in\langle n\rangle}
\]
is not symmetric.  Applying \cite[Corollary~7]{stroock-taniguchi},
we see that $x$ is $\tau^\prime$-regular.
\end{proof}

\begin{remark}\label{rem.hypoelliptic}
Since $\Delta_b$ is hypoelliptic
(\cite[Theorem~2.1]{dragomir-tomassini}),
that is, if
$\Delta_b v=g$ and $g\in C^{\infty}(U)$ then $v\in
C^{\infty}(U)$,
$u_f$ is a classical solution to the 
Dirichlet problem, namely it holds that
$u_f\in C^{\infty}(G)\cap C(\overline{G})$, 
$\Delta_b u_f=0$ and $u_f|_{\partial G}=f$. 
\end{remark}


\end{document}